\documentclass[11pt,a4paper]{amsart}
\usepackage{amssymb,xspace}
\usepackage{amstext}
\usepackage{tikz}
\usetikzlibrary{arrows,decorations.pathmorphing,backgrounds,fit,positioning,shapes.symbols,chains}
\theoremstyle{plain}
\usepackage{amsbsy,amssymb,amsfonts,latexsym}

\usepackage{enumerate}
\usepackage{graphics}

\date{\today}
\title[Multiple summing maps]{Multiple summing maps: coordinatewise summability, inclusion theorems and $p$-Sidon sets}
\author{Fr\'{e}d\'{e}ric Bayart} \address{Universit\'e Clermont Auvergne, CNRS, LMBP, F-63000 Clermont-Ferrand, France.} \email{frederic.bayart@uca.fr}



\newcommand{\veps}{\varepsilon}

\def\bi{\mathbf i}
\def\bj{\mathbf j}
\def\bk{\mathbf k}

\newtheorem{theorem}{Theorem}[section]

\newtheorem{lemma}[theorem]{Lemma}

\newtheorem{proposition}[theorem]{Proposition}

\newtheorem{corollary}[theorem]{Corollary}

{\theoremstyle{definition}}
{\theoremstyle{definition}}

{\theoremstyle{definition}}

{\theoremstyle{definition}}

{\theoremstyle{definition}}

{\theoremstyle{definition}\newtheorem{remark}[theorem]{Remark}}

\newtheorem*{DPS}{Theorem (Defant, Popa, Schwarting)}

\begin{document}

\begin{abstract}
We discuss the multiple summability of a multilinear map $T:X_1\times\cdots\times X_m\to Y$ when we have informations on the summability of the maps it induces on each coordinate. Our methods have applications to inclusion theorems for multiple summing multilinear mappings and to the product of $p$-Sidon sets.
\end{abstract}

\maketitle

\section{Introduction}
\subsection{Multiple and coordinatewise summability} Let $T:X\to Y$ be linear where $X$ and $Y$ are Banach spaces. For $r, p\geq 1$, we say that $T$ is $(r,p)$-summing if there exists a constant $C>0$ such 
that, for any sequence $x=(x_i)_{i\in\mathbb N}\subset X^\mathbb N$, 
$$\left(\sum_{i=1}^{+\infty} \|T(x_i)\|^r\right)^{\frac 1r}\leq Cw_p(x)$$
where the weak $\ell^p$-norm of $x$ is defined by
\begin{eqnarray*}
 w_p(x)&=&\sup_{\|x^*\|\leq 1}\left(\sum_{i=1}^{+\infty} |x^*(x_i)|^p\right)^{\frac 1p}.
\end{eqnarray*}
The theory of $(r,p)$-summing operators is very rich and very important in Banach space theory (see \cite{DJT} for details). In recent years, the interest moves to
multilinear maps. We start now from $m\geq 1$, $X_1,\dots,X_m$, $Y$ Banach spaces and $T:X_1\times\cdots\times X_m\to Y$ $m$-linear. Following
\cite{BPV04} and \cite{Mat03}, for $r\geq 1$ and $\mathbf p=(p_1,\dots,p_m)\in [1,+\infty)^m$, we say that $T$ is multiple $(r,\mathbf p)$-summing
if there exists a constant $C>0$ such that for all sequences $x(j)\subset X_j^{\mathbb N}$, $1\leq j\leq m$, 
$$\left(\sum_{\bi\in\mathbb N^m}\|T(x_\bi)\|^r\right)^{\frac 1r}\leq Cw_{p_1}(x(1))\cdots w_{p_m}(x(m))$$
where $T(x_\bi)$ stands for $T(x_{i_1}(1),\dots,x_{i_m}(m))$. The least constant $C$ for which the inequality holds is denoted by $\pi_{r,\mathbf p}^{\textrm{mult}}(T)$.
When all the $p_i$'s are equal to the same $p$, we will simply say that $T$ is multiple $(r,p)$-summing.

Even if the notion of multiple summing mappings was formalized only recently, its roots go back to an inequality of Bohnenblust and Hille 
appeared in 1931 (see \cite{BH31}). Using the reformulation of \cite{PGVil04}, this inequality says that every $m$-linear \emph{form} $T:X_1\times\cdots\times X_m\to\mathbb K$
is multiple $(2m/(m+1),1)$-summing. Observe that the restriction of $T$ to each $X_k$ (fixing the other coordinates) is, as all linear forms, $(1,1)$-summing. This motivates
the authors of \cite{DPS10} to study the following question: let $T:X_1\times\cdots\times X_m\to Y$ be $m$-linear and assume that the restriction of $T$
to each $X_k$ is $(r,p)$-summing (we will say that $T$ is \emph{separately summing}). Can we say something about the multiple $(s,t)$-summability of $T$?
The authors of \cite{DPS10} get a successful answer in the case $p=t=1$ (their results were later improved and simplified in \cite{PopSinna13} and in \cite{ANSS15}).
Precisely, they showed the following result:
\begin{DPS}
Let $T:X_1\times \cdots\times X_m\to Y$ be $m$-linear with $Y$ a cotype $q$ space. Let $r\in [1,q]$ and assume that $T$ is separately $(r,1)$-summing. Then $T$ is multiple $(s,1)$-summing, with
$$\frac 1s=\frac{m-1}{mq}+\frac{1}{mr}.$$
\end{DPS}
We intend in this paper to fill out the picture by allowing the full range of possible values for $t$ and $p$, namely $t\geq p\geq 1$.
The following result is a more readable corollary of our main theorems, Theorems \ref{thm:main1}, \ref{thm:main2}, \ref{thm:main3}, \ref{thm:opti1} ($p^*$ will denote the conjugate
exponent of $p$).

\begin{theorem}\label{thm:intro}
 Let $T:X_1\times\cdots\times X_m\to Y$ be $m-$linear with $Y$ a cotype $q$ space. Assume that $T$ is separately $(r,p)$-summing and let $t\geq p$. 
 \begin{itemize}
 \item If $\frac 1r+\frac{1}{p^*}-\frac{m}{t^*}>\frac 1q$, then 
 $T$ is multiple $(s,t)$-summing with 
 $$\frac 1s=\frac{m-1}{mq}+\frac 1{mr}+\frac{1}{mp^*}-\frac{1}{t^*}.$$
 \item If $0<\frac 1r+\frac{1}{p^*}-\frac{m}{t^*}\leq \frac 1q$, then $T$ is multiple $(s,t)$-summing with
 $$ \frac1s=\frac1r+\frac 1{p^*}-\frac{m}{t^*}.$$
 \end{itemize}
 When $1\leq p=t\leq 2$ and $q=2$, the above values of $s$ are optimal.
\end{theorem}

\subsection{Inclusion theorems} Our methods have other interesting consequences. A basic result in the theory of $(r,p)$-summing operators is the inclusion theorem: if $T\in\mathcal L(X,Y)$ is $(r,p)$-summing,
then it is also $(s,q)$-summing provided $s\geq r$ and $\frac1s-\frac 1q\leq \frac 1r-\frac 1p$. The proof of this result follows from a simple application of H\"older's inequality.

In the multilinear case, the situation seems more involved. Using probability in a clever way, P\'erez-Garc\'ia in \cite{PG04} succeeded to prove that if $T\in\mathcal L(X_1,\dots,X_m;Y)$ is
$(p,p)$-summing, $p\in [1,2)$, then it is also $(q,q)$-summing for $q\in [p,2)$. However, this result is not very helpful to provide inclusion theorems for $(r,p)$-summing multilinear maps as those
coming from the Bohnenblust-Hille inequality.

The next result seems to be a natural multilinear analogue to the linear inclusion theorem.  It already appeared in \cite[Proposition 3.4]{PSST17} in the particular case
where all the $p_i$ are equal, with a different proof. Its optimality will be discussed in Theorem \ref{thm:opti2}.

\begin{theorem}\label{thm:inclusion}
 Let $T:X_1\times\cdots\times X_m\to Y$ be $m$-linear, let $r,s\in [1,+\infty)$, $\mathbf p,\mathbf q\in[1,+\infty)^m$. Assume that $T$ is multiple $(r,\mathbf p)$-summing,
 that $q_k\geq p_k$ for all $k=1,\dots,m$ and that $\frac 1r-\sum_{j=1}^m\frac1{p_j}+\sum_{j=1}^m \frac1{q_j}>0$. Then $T$ is multiple $(s,\mathbf q)$-summing, with
 $$\frac 1s-\sum_{j=1}^m\frac1{q_j}=\frac 1r-\sum_{j=1}^m \frac 1{p_j}.$$
\end{theorem}

\subsection{Harmonic analysis}
A second application occurs in harmonic analysis. Let $G$ be a compact abelian group with dual group $\Gamma$. A subset $\Lambda$ of $\Gamma$ is called \emph{$p$-Sidon} ($1\leq p<2$) if there is a constant $\kappa>0$ such that each $f\in\mathcal C(G)$ with $\hat f$ supported on $\Lambda$ satisfies $\|\hat f\|_{\ell_p}\leq \kappa\|f\|_\infty.$ It is a classical result of Edwards and Ross \cite{ER74} (resp. Johnson and Woodward \cite{JW74}) that the direct product of two $1$-Sidon sets (resp. $m$ $1$-Sidon sets) is $4/3$-Sidon (resp. $2m/(m+1)$-Sidon). We generalize this to the product of $p$-Sidon sets. We need an extra assumption. A subset $\Lambda$ of $\Gamma$ is called a $\Lambda(p)$-set, $p\geq 1$, if for one $q\in [1,p)$ (equivalently, for all $q\in [1,p)$), there exists $\kappa>0$ such that, for all $f\in\mathcal C(G)$ with $\hat f$ supported on $\Lambda$, 
$$\|f\|_{L^p(G)}\leq \kappa \|f\|_{L^q(G)}.$$

\begin{theorem}\label{thm:sidon}
Let $G_1,\dots,G_m$, $m\geq 2$, be compact abelian groups with respective dual groups $\Gamma_1,\dots,\Gamma_m$. For $1\leq j\leq m$, let $\Lambda_j\subset\Gamma_j$ be a $p_j$-Sidon and $\Lambda(2)$-set. Then $\Lambda_1\times\cdots\times\Lambda_m$ is a $p$-Sidon set in $\Gamma_1\times\cdots\times \Gamma_m$ for
$$\frac1p=\frac 12+\frac{1}{2R}\textrm{ and }R=\sum_{k=1}^m \frac{p_k}{2-p_k}.$$
Moreover, this value of $p$ is optimal.
\end{theorem}

It is well known that any $1$-Sidon set is automatically a $\Lambda(p)$-set for all $p\geq 1$. It is not known whether all $p$-Sidon sets are $\Lambda(2)$ or not. We also get an analogous result for another natural generalization of $1$-Sidon sets, the so-called \emph{$p$-Rider sets}, without any extra assumption.

\smallskip

\textsc{Organization of the paper.} Section 2 is devoted to the introduction of some notations and definitions. We then give the statements of our main theorems (Theorems \ref{thm:main1}, \ref{thm:main2} and \ref{thm:main3}). These statements may look technical but we derive immediately from them several striking corollaries. We emphasize particularly Corollary \ref{cor:main} whose proof needs the three main results.

In Section 3, we prove several auxiliary results. They seem interesting for themselves; for instance, they are at the heart of the proof of Theorems \ref{thm:inclusion} and \ref{thm:sidon}. We apply these auxiliary results in the 
three next sections to the problems we have in mind: coordinatewise summability in Section 4, inclusion theorems in Section 5, and harmonic analysis in Section 6. Finally, in Section 7, we discuss
the optimality of our results.

\section{Preliminaries: notations and statements of the results}
\subsection{General statements}
We shall use the terminology and notations introduced in \cite{DPS10} and \cite{PopSinna13}. For Banach spaces $X_1,\dots,X_m$, $m\geq 2$, and a proper
subset $C$ of $\{1,\dots,m\}$, we write $X^C=\prod_{j\in C}X_j$ and identify in the obvious way $X_1\times\cdots\times X_m$ with $X^C\times X^{\overline{C}}$
where $\overline C$ denotes the complement of $C$ in $\{1,\dots,m\}$. With this identification, if $y\in X^C$ and $z\in X^{\overline{C}}$, then $x=(y,z)\in X_1\times\cdots\times X_m$. 
For $x\in X_1\times\cdots\times X_m$, we shall also denote by $x(C)$ its projection on $X^C$, so that we may write $x=(x(C),x(\overline{C}))$. We take the norm
on finite products of Banach spaces to be the maximum of the component norms; hence the identification is isometric. We shall abbreviate $x(\{k\})$ by $x(k)$, namely the $k$-th coordinate of $x\in X_1\times\cdots\times X_m$.

If $T:X_1\times\cdots\times X_m\to Y$ is $m$-linear and $z\in X^{\overline C}$, the map $T^{C}(z)$ defined on $X^C$ by $T^{C}(z)(x)=T(x,z)$ is clearly $|C|$-linear. For $r,p\geq 1$, we say that
$T$ is {\it coordinatewise multiple $(r,p)$-summing in the coordinates of $C$} provided $T^{C}(z)$ is multiple $(r,p)$-summing for all $z\in\overline{C}$. In that case, we shall denote
$$\|T^C\|_{CW(r,p)}=\sup\left\{\pi_{r,p}^{\textrm{mult}}(T^C(z));\ \|z\|_{X^{\overline C}}\leq 1\right\}.$$

Our first result deals with $(r,\mathbf p)$-multiple summing maps where $r$ does not exceed the cotype of the target space.

\begin{theorem}\label{thm:main1}
Let $m\geq 2$, let $\{1,\dots,m\}$ be the disjoint union of $n\geq 2$ non-empty subsets $C_1,\dots,C_n$, let $Y$ be a Banach space with cotype $q$ and let 
$r_1,\dots,r_n\in [1,q)$, $p_1,\dots, p_n\in [1,+\infty)$. Define
\begin{eqnarray*}
 \frac{1}{\gamma_k}&=&\frac1{r_k}-\sum_{j\neq k}\frac{|C_j|}{p_j^*}\times\frac{1-\frac q{r_k}-\frac{q|C_k|}{p_k^*}}{1-\frac q{r_j}-\frac{q |C_j|}{p_j^*}},\ k=1,\dots,n\\
 \frac1{\gamma_{k,l}}&=&\frac 1{r_k}-\sum_{j\neq k,l}\frac{|C_j|}{p_j^*}\times\frac{1-\frac q{r_k}-\frac{q|C_k|}{p_k^*}}{1-\frac q{r_j}-\frac{q |C_j|}{p_j^*}},\ k\neq l\in\{1,\dots,n\}\\
 R&=&\sum_{k=1}^n \frac{\gamma_k}{q-\gamma_k}\\
 s&=&\frac{qR}{1+R}\\
 q_j&=&p_k\ \textrm{provided }j\in C_k,\ j=1,\dots,m\\
 \mathbf q&=&(q_1,\dots,q_m).
\end{eqnarray*}
Let us also assume that, for all $k\neq l\in\{1,\dots,n\}$, $\gamma_k>0$, $0<\gamma_{k,l}\leq q$ and $\frac{|C_l|\gamma_{k,l}}{p_l^*}\leq 1$. Then all $m$-linear maps $T:X_1\times\cdots\times X_m\to Y$ which are
$(r_k,p_k)$-summing in the coordinates of $C_k$ for each $k=1,\dots,n$ are multiple $(s,\mathbf q)$-summing.
\end{theorem}

Our second result deals with $(r,\mathbf p)$-multiple summing maps with $r$ exceeding the cotype of the target space, but when we start from $(r_k,p_k)$-coordinatewise summability with $r_k\leq q$.

\begin{theorem}\label{thm:main2}
 Let $m\geq 2$, let $\{1,\dots,m\}$ be the disjoint union of $n\geq 2$ non-empty subsets $C_1,\dots,C_n$, let $Y$ be a Banach space with cotype $q$ and let 
$r_1,\dots,r_n\in [1,q)$, $p_1,\dots, p_n\in [1,+\infty)$. Define
\begin{eqnarray*}
 \frac{1}{\gamma_{k,J}}&=&\frac1{r_k}-\sum_{j\notin J\cup \{k\}}\frac{|C_j|}{p_j^*}\times\frac{1-\frac q{r_k}-\frac{q|C_k|}{p_k^*}}{1-\frac q{r_j}-\frac{q |C_j|}{p_j^*}},\ k=1,\dots,n,\ 
 J\subset\{1,\dots,n\}\backslash\{k\}\\
 q_j&=&p_k\ \textrm{provided }j\in C_k,\ j=1,\dots,m\\
 \mathbf q&=&(q_1,\dots,q_m). 
\end{eqnarray*}
Assume that there exists $J\subset\{1,\dots,n\}$ such that
\begin{enumerate}
 \item there exists $k_0\notin J$ with $\gamma_{k_0,J}\geq q$;
 \item For any $k,l\in\{1,\dots,n\}\backslash J$, $k\neq l$, $\gamma_{k,J\cup\{l\}}\in(0, q]$;
 \item For any $k,l\in\{1,\dots,n\}\backslash J$, $k\neq l$, $\frac{|C_l|\gamma_{k,J\cup\{l\}}}{p_l^*}\leq 1$.
\end{enumerate}
We finally set
$$\frac1s=\frac{1}{\gamma_{k_0,J}}-\sum_{j\in J}\frac{|C_j|}{p_j^*}$$
and assume that $s>0$.
Then all $m$-linear maps $T:X_1\times\cdots\times X_m\to Y$ which are
$(r_k,p_k)$-summing in the coordinates of $C_k$ for each $k=1,\dots,n$ are multiple $(s,\mathbf q)$-summing.
\end{theorem}

Our third result solves the case when one $r_k$ is greater than $q$.
\begin{theorem}\label{thm:main3}
 Let $m\geq 2$, let $\{1,\dots,m\}$ be the disjoint union of $n\geq 2$ non-empty subsets $C_1,\dots,C_n$, let $Y$ be a Banach space with cotype $q$ and let 
$r_1,\dots,r_n\in [1,+\infty)$, $p_1,\dots, p_n\in [1,+\infty)$. Assume that there exists $k\in\{1,\dots,n\}$ such that $r_k\geq q$. We set
$$\frac1 s=\frac{1}{r_k}-\sum_{j\neq k}\frac{|C_j|}{p_j^*}$$ 
and assume that $s>0$. Then all $m$-linear maps $T:X_1\times\cdots\times X_m\to Y$ which are
$(r_k,p_k)$-summing in the coordinates of $C_k$ for each $k=1,\dots,n$ are multiple $(s,\mathbf q)$-summing
where $\mathbf q$ is defined by $q_j=p_k$ for $j\in C_k$, $j=1,\dots,m$.
\end{theorem}

\subsection{Corollaries} The statement of Theorems \ref{thm:main1}, \ref{thm:main2} and \ref{thm:main3} may look complicated; this is due to their generality. In particular cases, they look nicer; they cover and extend many known statements.
We begin by assuming that $p_k=1$ for all $k\in\{1,\dots,n\}$.
\begin{corollary}
 Let $m\geq 2$, let $\{1,\dots,m\}$ be the disjoint union of $n\geq 2$ non-empty open subsets $C_1,\dots,C_n$, let $Y$ be a Banach space with cotype $q$ and let 
$r_1,\dots,r_n\in [1,q)$. Set 
$$R=\sum_{k=1}^n \frac{r_k}{q-r_k},\ s=\frac{qR}{1+R}.$$
Then all $m$-linear maps $T:X_1\times\dots\times X_m\to Y$ which are $(r_k,1)$-summing in the coordinates of $C_k$ for each $k=1,\dots,n$ are multiple $(s,1)$-summing.
\end{corollary}
This corollary is the main result of \cite{PopSinna13} which was itself an improved version of the main theorem of \cite{DPS10}. 
\begin{proof}
We may apply Theorem \ref{thm:main1}. Its assumptions are satisfied because $p_k^*=+\infty$.
\end{proof}
\begin{remark}
Observe that there is no restriction to assume $r_k< q$. Indeed, any linear map with value in a cotype $q$ space is always $(q,1)$-summing and we may apply Theorem \ref{thm:main3} to deduce that any multilinear map with value in a cotype $q$ space is always multiple $(q,1)$-summing, a result already observed in \cite[Theorem 3.2]{BPV04}
\end{remark}

Our second more appealing result happens when we start from a $(r_k,p_k)$-separately summing map (namely $|C_k|=1$ for all $k$) with $\frac1{r_k}-\frac{1}{p_k}=\theta\in (-\infty,0]$. In view of the inclusion theorem, this last assumption is not surprising. It implies that all the quotients
$$\frac{1-\frac q{r_k}-\frac{q|C_k|}{p_k^*}}{1-\frac q{r_j}-\frac{q |C_j|}{p_j^*}}$$
are equal to 1.

\begin{corollary}\label{cor:main}
Let $T:X_1\times\dots\times X_m\to Y$ with $Y$ a cotype $q$ space and $\mathbf p\in [1,+\infty)^m$. Assume that $T$ is $(r_k,p_k)$-summing in the $k$-th coordinate and that there exists $\theta<0$ such that $\frac1{r_k}-\frac1{p_k}=\theta$ for all $k$. Set
$$\frac 1\gamma=1+\theta-\sum_{k=1}^m\frac 1{p_k^*}.$$
\begin{enumerate}
\item If $\gamma\in (0,q)$, then $T$ is multiple $(s,\mathbf p)$-summing with
$$\frac 1s=\frac {m-1}{mq}+\frac{1}{\gamma m}.$$
\item If $\gamma\geq q$, then $T$ is multiple $(\gamma,\mathbf p)$-summing. 
\end{enumerate}
\end{corollary}
\begin{proof}
Suppose first that $\gamma\in (0,q)$. Then with the notations of Theorem \ref{thm:main1}, $\gamma_k=\gamma$ for all $k$ and $\frac{1}{\gamma_{k,l}}=\frac1\gamma+\frac 1{p_l^*}$
for all $k\neq l$. This implies that $r_k<q$ and $\frac1{\gamma_{k,l}}\geq\frac 1{p_l^*}$.
Hence the assumptions of Theorem \ref{thm:main1} are satisfied and this leads to (1). To prove (2), we suppose first that $r_k<q$ for all $k$.
Let $J$ be a maximal set of $\{1,\dots,n\}$ such that there exists $k_0\notin J$ with $\gamma_{k_0,J}\geq q$. Such a set does exist since $\gamma_{1,\varnothing}=\gamma\geq q$
and $\gamma_{k,\{1,\dots,n\}\backslash \{k\}}=r_k<q$ for all $k$. This couple $J$ and $k_0$ being fixed, we may observe that for all $k,l\in\{1,\dots,n\}\backslash J$, $k\neq l$, 
$\gamma_{k,J\cup\{l\}}<q$ (otherwise $J$ would not be maximal) and 
$$\frac1{\gamma_{k,J\cup\{l\}}}=\frac 1{\gamma_{k,J}}+\frac 1{p_l^*}\geq \frac 1\gamma+\frac 1{p_l^*}\geq\frac 1{p_l^*}.$$
Thus we may apply Theorem \ref{thm:main2}. Finally, if $r_k\geq q$ for some $k$, then the result follows from Theorem \ref{thm:main3}.
\end{proof}

In turn, this last corollary implies several interesting results. First, half of Theorem \ref{thm:intro} may be deduced easily from it.
\begin{proof}[Proof of Theorem \ref{thm:intro} (without optimality)]
 Assume first that $t=p$. Then the conclusion follows directly from Corollary \ref{cor:main} with $r_k=r$ and $p_k=p$ for all $k$. Suppose now that $t>p$. Then, by the inclusion theorem
 for linear maps, $T$ is separately $(\rho,t)$-summing for $\frac1\rho=\frac 1r+\frac 1t-\frac 1p$. We conclude again by an application of Corollary \ref{cor:main} with $r_k=\rho$
 and $p_k=t$ for all $k$.
\end{proof}
We may also deduce from Corollary \ref{cor:main} a result of Praciano-Pereira \cite{Pr81} and Dimant/Sevilla-Peris \cite{DiSP14} which is an $m$-linear version of a famous bilinear inequality of Hardy and Littlewood \cite{HL34}. We state it in the spirit of \cite{PGVil04}.

\begin{corollary}\label{cor:praciano}
 Let $T:X_1\times\dots\times X_m\to \mathbb C$ be $m$-linear and let $\mathbf p=(p_1,\dots,p_m)\in [1,+\infty)^m$. Set
 $$\frac1\gamma=1-\sum_{k=1}^m \frac 1{p_k^*}.$$
 \begin{enumerate}
 \item If $\gamma\in (0,2)$ then $T$ is multiple $(s,\mathbf p)$-summing with 
 $$\frac 1s=\frac {m-1}{2m}+\frac1{m\gamma}.$$
 \item If $\gamma\geq 2$, then $T$ is multiple $(\gamma,\mathbf p)$-summing.
 \end{enumerate}
\end{corollary}

\begin{proof}
This follows immediately from Corollary \ref{cor:main} since any linear form is $(p,p)$-summing. 
\end{proof}
Observe finally that Theorem \ref{thm:intro} extends also Theorem 1.2 of \cite{DiSP14}.

\smallskip

\noindent \textsc{Notations.} Part of the notations we shall use was already introduced at the beginning of this section. We shall also denote by $(e_i)_{i\in\mathbb N}$ the standard basis of $\ell_p$ and $e_\bi$, $\bi\in\mathbb N^m$, will mean $(e_{i_1}(1),\dots,e_{i_m}(m))$ where $(e_i(j))_i$ is a copy of $(e_i)_i$. For $u\in\prod_{j=1}^m \ell_{p_j}$, $\bi\in\mathbb N^m$ and $\alpha\in\mathbb R$, $u_\bi$ will stand for $u_{i_1}(1)\times \cdots\times u_{i_m}(m)$ and $u_\bi^\alpha$ for $u_{i_1}(1)^\alpha \times \cdots\times u_{i_m}(m)^\alpha$. As indicated above, if $(a_\bi)_{\bi\in\mathbb N^m}$ is a sequence indexed by $\mathbb N^m$ and $C\subset \{1,\dots,m\}$, we shall identify $\bi$ with $\bj,\bk$ with $\bj=\bi(C)$, $\bk=\bi(\bar C)$ so that we shall write $a_\bi=a_{\bj,\bk}$.

\section{Useful lemmas}
\subsection{Coefficients of non-negative $m$-linear forms}
We shall need the following non-negative version of a theorem of Praciano-Pereira \cite{Pr81}. It already appears in \cite{Lac13} for bilinear forms.
\begin{proposition}\label{prop:praciano}
Let $m\geq 1$, $1\leq p_1,\dots,p_m\leq+\infty$ and $A:\ell_{p_1}\times\cdots\times\ell_{p_m}\to \mathbb C$ be a non-negative $m$-linear form. Then 
$$\left(\sum_{\bi\in\mathbb N^m}A(e_\bi)^\rho\right)^{1/\rho}\leq\|A\|$$
provided $\rho^{-1}=1-\sum_{j=1}^m p_j^{-1}>0$.
\end{proposition}
Here, non-negative simply means that for any $\bi\in\mathbb N^m$, $A(e_\bi)\geq 0$. 
\begin{proof}
We shall give a proof by induction on $m$. Our main tool is the following factorization result of Schep \cite{Schep84} which extends to multilinear maps a result of Maurey \cite{Mau74}.
\begin{lemma}\label{lem:schep}
Let $B:\ell_{p_1}\times\cdots\times\ell_{p_m}\to \ell_q$ be a non-negative $m$-linear map such that $r\geq \max(q,1)$ with $r^{-1}=p_1^{-1}+\dots+p_m^{-1}$. Then there exist a non-negative $\phi\in\ell_s$ with $s^{-1}=q^{-1}-r^{-1}$ and a non-negative $m$-linear map $C:\ell_{p_1}\times\cdots\times\ell_{p_m}\to \ell_r$ such that $B=M_\phi C$ where $M_\phi$ is the operator of multiplication by $\phi$. Moreover, $\|B\|=\inf \|\phi\|_s\|C\|$ where the infimum is taken over all possible factorizations.
\end{lemma}
Let us come back to the proof of Proposition \ref{prop:praciano}. The result is clear for $m=1$ (it does not require positivity) and let us assume that it is true for $m$-linear forms, $m\geq 1$. Let $A:\ell_{p_1}\times\cdots\times\ell_{p_{m+1}}\to \mathbb C$
be a non-negative $(m+1)$-linear form. It defines a bounded $m$-linear map $B:\ell_{p_1}\times\cdots\times\ell_{p_m}\to \ell_{p_{m+1}^*}$ by $\langle e_j,B(x)\rangle=A(x,e_j)$. By Lemma \ref{lem:schep}, 
$B$ factors through $\ell_r$,
$r^{-1}=p_1^{-1}+\dots+p_m^{-1}$; namely we may write $B=M_\phi C$ with $\phi\in\ell_s$, $s^{-1}=1-p_1^{-1}-\dots-p_{m+1}^{-1}$ and $C:\ell_{p_1}\times\cdots\times \ell_{p_m}\to\ell_{r}$ a non-negative continuous
$m$-linear map. Thus, writing $a_{\bi,j}=A(e_\bi,e_j)=\langle e_j,B(e_\bi)\rangle$, $c_{\bi,j}=\langle e_j,C(e_\bi)\rangle$, $\bi\in\mathbb N^m$, $j\in\mathbb N$, we get
\begin{eqnarray*}
\left(\sum_{j\in\mathbb N}\sum_{\bi\in\mathbb N^m}a_{\bi,j}^s\right)^{1/s} &=&\left(\sum_{j\in\mathbb N} \phi_j^s\sum_{\bi\in\mathbb N^m}c_{\bi,j}^s\right)^{1/s}\\
&\leq&\|\phi\|_s\sup_{j\in\mathbb N} \left(\sum_{\bi\in\mathbb N^m}c_{\bi,j}^s\right)^{1/s}.
\end{eqnarray*}
Define now $C_j:\ell_{p_1}\times\cdots\times\ell_{p_m}\to \mathbb C$ by $C_j(x)=\langle e_j,C(x)\rangle$. Then $C_j$ is a bounded non-negative $m$-linear form with $\|C_j\|\leq \|C\|$, and by the induction hypothesis, since $s\geq t$ where $t^{-1}=1-p_1^{-1}-\dots-p_m^{-1}$, we have
$$\left(\sum_{\bi\in\mathbb N^m} c_{\bi,j}^s\right)^{1/s}\leq \|C\|.$$
The result now follows by taking the infimum over all possible factorizations of $A$.
\end{proof}

\begin{remark}
 The example of $A(x(1),\dots,x(m))=\sum_{i=1}^n x_i(1)\cdots x_i(m)$ shows that the constant $\rho$ in Proposition \ref{prop:praciano} is optimal.
\end{remark}

\subsection{An abstract Hardy-Littlewood method}
To prove their bilinear inequality on $\ell_p$-spaces in \cite{HL34}, Hardy and Littlewood have introduced a methode to go from $\ell_p$ to $c_0$ and back again.
This was performed several times later (for instance in \cite{Pr81}, \cite{BayPelSeo13} or \cite{DiSP14}). We shall develop
here an abstract version of this machinery, first in the bilinear case and then in the $m$-linear one.
\begin{lemma}\label{lem:hardylittlewood}
Let $m_1,m_2\geq 1$, $p_1,p_2,q\in [1,+\infty)$, $(a_{\bi,\bj})_{\bi\in\mathbb N^{m_1},\ \bj\in \mathbb N^{m_2}}$ a sequence of non-negative real numbers.
Assume that there exists $\kappa>0$ and $0<\alpha,\beta\leq q$ such that
\begin{itemize}
\item for all $u\in \prod_{j=1}^{m_1}B_{\ell_{p_1}}$, 
$$\left(\sum_{\bj\in\mathbb N^{m_2}}\left(\sum_{\bi\in \mathbb N^{m_1}}u_\bi^q a_{\bi,\bj}^q\right)^{\alpha/q}\right)^{1/\alpha}\leq \kappa;$$
\item for all $v\in \prod_{j=1}^{m_2}B_{\ell_{p_2}}$, 
$$\left(\sum_{\bi\in\mathbb N^{m_1}}\left(\sum_{\bj\in \mathbb N^{m_2}}v_\bj^q a_{\bi,\bj}^q\right)^{\beta/q}\right)^{1/\beta}\leq \kappa.$$
\end{itemize}
Then 
$$\left(\sum_{{\bj\in \mathbb N^{m_2}}}\left(\sum_{\bi\in\mathbb N^{m_1}}a_{\bi,j}^q\right)^{\gamma/q}\right)^{1/\gamma}\leq \kappa$$
where
$$\frac1\gamma=\frac1\alpha-\frac {m_1}{p_1}\left(\frac{1-\frac q\alpha-\frac {m_2q}{p_2}}{1-\frac {q}\beta-\frac{m_1q}{p_1}}\right)$$
provided $\gamma>0$, $\frac{m_1 \alpha}{p_1}\leq 1$ and $\frac{m_2 \beta}{p_2}\leq 1$.
\end{lemma}
\begin{proof}
For $\bj\in \mathbb N^{m_2}$, we denote $S_\bj=\left(\sum_{\bi\in\mathbb N^{m_1}}a_{\bi,\bj}^q\right)^{1/q}$. Let also $\theta>0$ with $m_2/\theta<1$ and let $1/\rho=1-m_2/\theta$. 
For any $\gamma\in\mathbb R$, we may write
\begin{eqnarray*}
 \sum_{\bj}S_\bj^\gamma&=&\sum_{\bj}S_\bj^{\left(\frac\gamma\rho\right)\rho}\\
&\leq&\left(\sup_{w\in\prod_{j=1}^{m_2}B_{\ell_\theta}}\sum_\bj w_\bj S_\bj^{\frac{\gamma}{\rho}}\right)^{\rho}
\end{eqnarray*}
where we have used Proposition \ref{prop:praciano}. We then set $\gamma'=\gamma/\rho$ and we write
for $w\in\prod_{j=1}^{m_2}B_{\ell_\theta}$,
\begin{eqnarray*}
\sum_{\bj}w_\bj S_\bj^{\gamma'}&=&\sum_\bj  w_\bj S_\bj^{\gamma'-q}\sum_{\bi}a_{\bi,\bj}^{q}\\
&=&\sum_{\bi}\sum_{\bj}\frac{w_\bj a_{\bi,\bj}^q}{S_\bj^{q-\gamma'}}\\
&\leq&\sum_{\bi}\left(\sum_\bj \frac{a_{\bi,\bj}^q}{S_\bj^{(q-\gamma')s}}\right)^{1/s}\left(\sum_\bj w_\bj^{s^*}a_{\bi,\bj}^q\right)^{1/s^*}\\
&\leq&\left(\sum_{\bi}\left(\sum_\bj \frac{a_{\bi,\bj}^q}{S_\bj^{(q-\gamma')s}}\right)^{t/s}\right)^{1/t}\left(\sum_{\bi}\left(\sum_\bj w_{\bj}^{s^*}a_{\bi,\bj}^q\right)^{t^*/s^*}\right)^{1/t^*}
\end{eqnarray*}
where $(s,s^*)$ and $(t,t^*)$ are two couples of conjugate exponents such that $t^*/s^*=\beta/q$. 
Now, $w^{s^*/q}$ belongs to $\prod_{j=1}^{m_2}B_{\ell_{\theta q/s^*}}$. Thus, if we can set $\theta=\frac{p_2 s^*}q$,
then we can deduce that
$$\sum_{\bj} w_\bj S_{\bj}^{\gamma'}\leq \kappa^{1/t^*}\left(\sum_{\bi}\left(\sum_{\bj} \frac{a_{\bi,\bj}^q}{S_{\bj}^{(q-\gamma')s}}\right)^{t/s}\right)^{1/t}.$$
We then apply Proposition \ref{prop:praciano} to the $m$-linear form defined on $\prod_{k=1}^{m_1} \ell_{\omega}$ by
$A(e_\bi)=\sum_{\bj\in \mathbb N^{m_2}}\frac{a_{\bi,\bj}^q}{S_{\bj}^{(q-\gamma')s}}$ where 
$$\frac {m_1}\omega=1-\frac st$$
(this requires $s\leq t$). We obtain
$$\sum_{\bi\in\mathbb N^{m_1}}\left(\sum_{\bj\in \mathbb N^{m_2}}\frac{a_{\bi,\bj}^q}{S_\bj^{(q-\gamma')s}}\right)^{t/s}\leq\left( \sup_{y\in\prod_{k=1}^{m_1} B_{\ell_\omega}}
\sum_{\bi\in\mathbb N ^{m_1}}y_{\bi}\sum_{\bj\in \mathbb N^{m_2}}\frac{a_{\bi,\bj}^q}{S_\bj^{(q-\gamma')s}}\right)^{t/s}.$$
Fix now $y\in\prod_{k=1}^{m_1} B_{\ell_\omega}$ and let us apply another time H\"older's inequality with $r$ satisfying $(q-\gamma')sr=q$. We get
\begin{eqnarray*}
\sum_{\bi}y_{\bi}\sum_{\bj}\frac{a_{\bi,\bj}^q}{S_{\bj}^{(q-\gamma')s}}&=&\sum_\bj\sum_{\bi}
y_{\bi}\frac{a_{\bi,\bj}^q}{S_\bj^{(q-\gamma')s}}\\
&\leq&\sum_\bj \underbrace{\left(\sum_{\bi}\frac{a_{\bi,\bj}^q}{S_\bj^q}\right)^{1/r}}_{=1}\left(\sum_{\bi}y_{\bi}^{r^*}a_{\bi,\bj}^q\right)^{1/r^*}.
\end{eqnarray*}
We may then conclude provided 
$$\frac{r^*}{\omega}=\frac q{p_1}\textrm{ and }\frac 1{r^*}=\frac\alpha q.$$
All the conditions imposed on $r,s,t$ and $\omega$ fix the value of $\gamma'$. Indeed, we get successively
$$\frac1\omega=\frac{q}{p_1r^*}=\frac{\alpha}{p_1},\ s=\left(1-\frac{m_1\alpha}{p_1}\right)t,$$
$$t=\frac{1-\frac{\beta}q\left(1-\frac{m_1\alpha}{p_1}\right)}{\left(1-\frac{m_1\alpha}{p_1}\right)\left(1-\frac\beta q\right)} \textrm{ since } \frac{t^*}{s^*}=\frac{\beta}q,$$
$$s=\frac{1-\frac{\beta}q\left(1-\frac{m_1\alpha}{p_1}\right)}{1-\frac\beta q},\ \gamma'=q\left(1-\frac{\left(1-\frac\alpha q\right)\left(1-\frac\beta q\right)}{1-\frac\beta q+\frac{\alpha \beta m_1}{p_1q}}\right).$$
We may then compute $\gamma$ by checking that
\begin{eqnarray*}
\frac1\rho&=&1-\frac{m_2 q}{p_2s^*}\\
&=&1-\frac{\frac{\alpha\beta m_1m_2}{p_1p_2}}{1-\frac\beta q+\frac{\alpha\beta m_1}{p_1q}}.
\end{eqnarray*}
We finally deduce that
\begin{eqnarray*}
\gamma&=&\gamma'\rho\\
&=&\alpha\frac{1-\frac\beta q+\frac{\beta m_1}{p_1}}{1-\frac\beta q+\frac{\alpha\beta m_1}{p_1 q}-\frac{\alpha\beta m_1m_2}{p_1p_2}}
\end{eqnarray*}
which leads  to
\begin{eqnarray*}
\frac 1\gamma&=&\frac{1-\frac\beta q+\frac{\alpha\beta m_1}{p_1q}-\frac{\alpha \beta m_1m_2}{p_1p_2}}
{\alpha\left(1-\frac\beta q+\frac{\beta m_1}{p_1}\right)}\\
&=&\frac 1\alpha+\frac{m_1}{p_1}\frac{\frac{\alpha\beta}q-\beta-\frac{m_2 \alpha\beta}{p_2}}{\alpha\left(1-\frac\beta q+\frac{\beta m_1}{p_1}\right)}\\
&=&\frac 1\alpha-\frac{m_1}{p_1}\times \frac{1-\frac q\alpha-\frac{m_2}{p_2 q}}{1-\frac q\beta -\frac{m_1}{p_1 q}}.
\end{eqnarray*}
We verify now that our applications of H\"older's inequality and Proposition \ref{prop:praciano} were legitimate. It is clear that $s,r\geq 1$. Since $$\frac{s}{t}={1-\frac{m_1\alpha}{p_1}}$$
we also have $1\leq s\leq t$. In particular, our application of Proposition \ref{prop:praciano} to $\prod_{k=1}^{m_2}\ell_\omega$ was possible. Finally, our first application of this proposition requires that $\rho>0$, namely
$$\frac{\alpha\beta m_1m_2}{p_1p_2}\leq 1-\frac\beta q+\frac{\alpha\beta m_1}{p_1q}
\iff \frac{\alpha m_1}{p_1}\left(\frac{\beta m_2}{p_2}-\frac \beta q\right)\leq 1-\frac\beta q.$$
It is easy to check that this last inequality is satisfied provided $\alpha m_1\leq p_1$, $\beta m_2\leq p_2$ and $\beta\leq q$.
\end{proof}

The following proposition is the main step towards the proof of our main results. It is an $n$-linear version of the previous lemma.
\begin{proposition}\label{prop:main}
Let $\mathbf q\in [1,+\infty)^m$. Let $(C_1,\dots,C_n)$ be a partition of $\{1,\dots,m\}$ into non-empty open subsets and let us assume that there exists $\mathbf p\in[1,+\infty)^n$ such that, for any $l\in\{1,\dots,n\}$ and any $k\in C_l$, $q_k=p_l$. Let also $(a(\bi))_{\bi\in\mathbb N^m}$ be a sequence of non-negative real numbers. Assume that there exist $\kappa>0$, $0< r_1,\dots,r_n\leq q$ such that
for all $k\in\{1,\dots,n\}$, for all sequence $v\in \prod_{l\neq k} \prod_{j\in C_l}B_{\ell_{p_l}}$,
$$\sum_{\bi\in\mathbb N^{C_k}}\left(\sum_{\bj\in \mathbb N^{\overline{C_k}}}v_\bj^q a_{\bi,\bj}^q\right)^{\frac{r_k}q}\leq \kappa^{r_k}.$$
Define, for all $k\neq l$,
\begin{eqnarray*}
\frac1{\gamma_k}&=&\frac{1}{r_k}-\sum_{j\neq k}\frac{|C_j|}{p_j}\left(\frac{1-\frac{q}{r_k}-\frac{|C_k|q}{p_k}}{1-\frac{q}{r_j}-\frac{|C_j|q}{p_j}}\right)\\
\nonumber
\frac1{\gamma_{k,l}}&=&\frac{1}{r_k}-\sum_{j\neq k,l}\frac{|C_j|}{p_j}\left(\frac{1-\frac{q}{r_k}-\frac{|C_k|q}{p_k}}{1-\frac{q}{r_j}-\frac{|C_j|q}{p_j}}\right)
\end{eqnarray*}
Then, for all $k\in\{1,\dots,n\}$, 
\begin{equation}\label{eq:mainprop1}
\left(\sum_{\bi\in\mathbb N^{C_k}}\left(\sum_{\bj\in\mathbb N^{\overline{C_k}}} a_{\bi,\bj}^q\right)^{\frac{\gamma_k}q}\right)^{\frac 1{\gamma_k}}\leq \kappa
\end{equation}
provided, for all $k\neq l$, $\gamma_k>0$, $\gamma_{k,l}\in (0,q]$ and $\frac{|C_l|\gamma_{k,l}}{p_l}\leq 1$.
\end{proposition}
\begin{proof}
The proof is done by induction on $n$. For $n=1$, there is nothing to prove (the inner sum does not appear) and the case $n=2$ is the content of Lemma \ref{lem:hardylittlewood}. So, let us assume that the result is true for $n-1\geq 2$ and let us prove it for $n$. We fix some $l\in\{1,\dots,n\}$
and some $w\in\prod_{j\in C_l} B_{\ell_{p_l}}$. We then  define, for $\bi\in\mathbb N^{\overline{C_l}}$, 
$$b_l(\bi)=\left(\sum_{\bj\in\mathbb N^{C_l}}w_{\bj}^qa_{\bi,\bj}^q\right)^{\frac 1q}.$$
Our assumption implies that, for $k\neq l$,
$$\sum_{\bi\in\mathbb N^{C_k}}\left(\sum_{\bj\in\mathbb N^{\overline{C_k\cup C_l}}} v_\bj^q b_l(\bi,\bj)^q\right)^{\frac{r_k}q}\leq \kappa^{r_k}$$
where $v$ is any element of $\prod_{s\neq k,l}\prod_{j\in C_s} B_{\ell_{p_s}}$. We may thus apply the induction hypothesis to get that, for any $k\neq l$
$$\sum_{\bi\in\mathbb N^{C_k}}\left(\sum_{\bj\in\mathbb N^{\overline{C_k\cup C_l}}}b_l(\bi,\bj)^q\right)^{\gamma_{k,l}}\leq \kappa^{\gamma_{k,l}}.$$
We then set, for $\bi\in\mathbb N^{C_k}$ and $\bj\in \mathbb N^{C_l}$, 
$$c_{k,l}(\bi,\bj)=\left(\sum_{\bk\in\mathbb N^{\overline{C_k\cup C_l}}} a_{\bi,\bj,\bk}^q\right)^{\frac 1q}$$
so that our inequality becomes
$$\sum_{\bi\in\mathbb N^{C_k}}\left(\sum_{\bj\in\mathbb N^{C_l}} w_\bj^q c_{k,l}(\bi,\bj)^q\right)^{\frac{\gamma_{k,l}}q}\leq \kappa^{\gamma_{k,l}}$$
which is satisfied for all $w\in\prod_{j\in C_l} B_{\ell_{p_l}}$. But of course, we can exchange the role played by $k$ and $l$ and we also have
$$\sum_{\bj\in\mathbb N^{C_l}}\left(\sum_{\bi\in\mathbb N^{C_k}} w_\bi^q c_{k,l}(\bi,\bj)^q\right)^{\frac{\gamma_{l,k}}q}\leq \kappa^{\gamma_{l,k}}
$$
for all $w\in\prod_{j\in C_k} B_{\ell_{p_k}}$. We now apply Lemma \ref{lem:hardylittlewood} to find that \eqref{eq:mainprop1} is satisfied with 
$$\frac 1{\gamma_k}=\frac{1}{\gamma_{k,l}}-\frac{|C_l|}{p_l}\left(\frac{1-\frac{q}{\gamma_{k,l}}-\frac{|C_k|q}{p_k}}{1-\frac q{\gamma_{l,k}}-\frac{|C_l| q}{p_l}}\right).$$
It remains to verify that this is the expected value of $\gamma_k$. This follows from
\begin{eqnarray*}
1-\frac q{\gamma_{k,l}}-\frac{|C_k| q}{p_k}&=&1-\frac{q}{r_k}-q\sum_{j\neq k,l}\frac{|C_j|}{p_j}
\left(\frac{1-\frac{q}{r_k}-\frac{|C_k|q}{p_k}}{1-\frac{q}{r_j}-\frac{|C_j|q}{p_j}}\right)-\frac{|C_k| q}{p_k}\\
&=&\left(1-\frac{q}{r_k}-\frac{|C_k| q}{p_k}\right)\left(1-\sum_{j\neq k,l}\frac q{1-\frac{q}{r_j}-\frac{|C_j| q}{p_j}}\right)
\end{eqnarray*}
and from the symmetric computation involving $\gamma_{l,k}$.
\end{proof}

\subsection{A mixed-norm inequality}

We finally need a last result which is a combination of a mixed-norm H\"older inequality (see \cite{BenPan61}) and an inequality due to Blei (see \cite{Bl79}). It appears in \cite{PopSinna13}. Let $(M_j,\mu_j)$ be $\sigma$-finite measure spaces for $j=1,\dots,n$ and introduce the product measure spaces $(M^n,\mu^n)$ and $(M_k^n,\mu_j^n)$ by
$$M^n=\prod_{k=1}^n M_k,\ \mu^n=\prod_{k=1}^n \mu_k,\ M_j^n=\prod_{\substack{k=1\\ k\neq j}}^n M_k,\ 
\mu_j^n=\prod_{\substack{k=1 \\ k\neq j}}^n \mu_k.$$
\begin{lemma}\label{lem:popa}
Let $q>0$, $n\geq 2$ and $r_1,\dots,r_n\in (0,q)$. If $h\geq 0$ is $\mu^n$-measurable, then
$$\left(\int_{M_n}h^Q d\mu^n\right)^{\frac 1Q}\leq \prod_{j=1}^n \left(\int_{M_j}\left(\int_{M_j^n}h^q d\mu_j^n\right)^{\frac{r_j}q}d\mu_j\right)^{\frac{1}{R(q-r_j)}}$$
where $R=\sum_{j=1}^n \frac{r_j}{q-r_j}$ and $Q=\frac{qR}{1+R}$.
\end{lemma}

\section{Proof of the main results}
\begin{proof}[Proof of Theorem \ref{thm:main1}]
Let, for $1\leq j\leq m$,  $x(j)=(x_i(j))_{i\in\mathbb N}\subset X_j^\mathbb N$ with $w_{q_j}(x(j))\leq 1$. We set $a_\bi=\|T(x_\bi)\|$ for $\bi\in\mathbb N^m$ and we intend to show that the 
assumptions of Proposition \ref{prop:main} are satisfied. So, let $k\in\{1,\dots,n\}$. For $l\neq k\in\{1,\dots,n\}$ and $u\in C_l$, we consider a sequence
$v(u)\in B_{\ell_{p_l^*}}=B_{\ell_{q_j^*}}$ and we set $y(u)=(v_i(u) x_i(u))_{i\in\mathbb N}$ so that $w_1(y(u))\leq 1$. Writing $\overline{C_k}=\{u_1,\dots,u_s\}$ and picking $\bj\in \mathbb N^{\overline{C_k}}$, 
we set $y_\bj=y_\bj(\overline{C_k})=(y_{j_1}(u_1),\dots,y_{j_s}(u_s))$, so that
$$\sum_{\bi\in\mathbb N^{C_k}}\left(\sum_{\bj\in\mathbb N^{\overline{C_k}}} v_\bj^q a_{\bi,\bj}^q\right)^{\frac{r_k}q}=\sum_{\bi\in\mathbb N^{C_k}}\left(\sum_{\bj\in\mathbb N^{\overline{C_k}}}\|T(x_\bi(C_k),y_\bj(\overline{C_k}))\|^q\right)^{\frac{r_k}q}.$$
Since $Y$ has cotype $q$, and using Kahane's inequalities, there exists a constant $A_k$ (depending only on $r_k$, on $|\overline{C_k}|$ and on the cotype $q$ constant of $Y$) such that
$$\sum_{\bi\in\mathbb N^{C_k}}\left(\sum_{\bj\in\mathbb N^{\overline{C_k}}} v_\bj^q a_{\bi,\bj}^q\right)^{\frac{r_k}q}\leq A_k\int_\Omega\sum_{\bi\in\mathbb N^{C_k}} \|T(x_{\bi}(C_k),y(\omega))\|^{r_k}d\mathbb P(\omega)$$
where $y(\omega)=\left(\sum_{i=1}^{+\infty}\veps_{j,i}(\omega)y_i(j)\right)_{j\in\overline{C_k}}$
and $(\veps_{j,i})_{j\in\overline{C_k},\ i\in\mathbb N}$ are sequences of independent Bernoulli variables on the same probability space $(\Omega,\mathcal A,\mathbb P)$. Recall that $|\veps_{j,i}(\omega)|\leq 1$, for any $j\in\overline{C_k}$ and any $i\in\mathbb N$. Therefore, 
\begin{eqnarray*}
\left\|y(j,\omega)\right
\|_{X_j}&=&\sup_{x^*\in B_{X^*_j}}\left|\left\langle x^*,\sum_{i=1}^{+\infty}\veps_{j,i}(\omega)y_i(j)\right\rangle\right|\\
&\leq&w_1(y(j))\leq 1.
\end{eqnarray*}
Since $T$ is coordinatewise multiple summing in the coordinates of $C_k$, this yields 
$$\sum_{\bi\in\mathbb N^{C_k}}\left(\sum_{\bj\in\mathbb N^{\overline{C_k}}}v_\bj^q a_{\bi,\bj}^q\right)^{\frac{r_k}q}\leq A_k^{r_k} \|T^{C_k}\|_{CW(r_k,p_k)}^{r_k}.$$
Setting $\kappa=\max_k A_k \|T^{C_k}\|_{CW(r_k,p_k)}$, we may apply Proposition \ref{prop:main} which yields, for any $k\in\{1,\dots,m\}$, 
$$\left(\sum_{\bi\in\mathbb N^{C_k}}\left(\sum_{\bj\in\mathbb N^{\overline{C_k}}}\|T(x_\bi(C_k),x_\bj(\overline{C_k})\|^q\right)^{\frac{\gamma_k}q}\right)^{\frac{1}{\gamma_k}}\leq \kappa.$$
We conclude by Lemma \ref{lem:popa}.
\end{proof}

\begin{remark}
We have $A_k\leq \left(C_q(Y)K_{r_k,q}\right)^{|\overline{C_k}|}$ where $C_q(Y)$ is the cotype $q$ constant of $Y$ and $K_{r_k,q}$ is the constant appearing in Kahane's inequality between the $L^{r_k}$ and the $L^q$-norms. Hence, we have shown that 
$$\pi_{r,\mathbf q}^{\mathrm{mult}}(T)\leq \sup_{k=1,\dots,m}\left\{\left(C_q(Y)K_{r_k,q}\right)^{|\overline{C_k}|}\|T^{C_k}\|_{CW(r_k,p_k)}\right\}.$$
\end{remark}

\smallskip

The forthcoming lemma will be uselful for $(r,\mathbf p)$-multiple summing maps with $r$ greater than the cotype of the target space. It is inspired by the proof
of Theorem 1.2 of \cite{DiSP14}.

\begin{lemma}\label{lem:displike}
 Let $T:X_1\times\cdots\times X_m\to Y$ be $m$-linear with $Y$ a cotype $q$ space. 
 Let $\mathbf q\in [1,+\infty)^m$ and $C\subset \{1,\dots,m\}$. We define $\mathbf t\in [1,+\infty)^C$ by
 $t_k=q_k$ for all $k\in C$. Let finally $s,r\in [1,+\infty)$ satisfying
 $$\frac1r=\frac 1s+\sum_{j\in \bar C}\frac{1}{q_j^*},$$
 $r\geq q$ and $s\geq q_k$ for all $k\in\{1,\dots,m\}$. Then there exists $\kappa>0$ such that
 $$\pi_{s,\mathbf q}^{\mathrm{mult}}(T)\leq \kappa \sup\left\{\pi_{r,\mathbf t}^{\mathrm{mult}}\big(T^{C}(z)\big);\ \|z\|_{X^{\bar C}}\leq 1\right\}.$$
\end{lemma}
If all the $q_k$ are equal to the same $p$, the conclusion takes the more pleasant form: 
$$\pi_{s,q}^{\textrm{mult}}(T)\leq \kappa \|T^C\|_{CW(r,t)},\ \frac1r=\frac 1s+\frac{|\bar C|}{p^*}.$$
Note that we require now coordinatewise summability only in the coordinates of $C$ (and nothing on $\bar C$). 
But now, we start with $(r,\mathbf t)$-summability with $r$ greater than the cotype of the target space.
\begin{proof}
 Let $x$ belong to $\prod_{k=1}^m B_{\ell_{q_k}^w(X_k)}$. We write 
 \begin{equation}\label{eq:displike}
  \left(\sum_{\bi\in\mathbb N^m}\|T(x_\bi)\|^s\right)^{1/s}=\left(\sum_{\bi\in\mathbb N^{\bar C}}\|y_\bi\|_{\ell_s(Y)}^s\right)^{1/s}
 \end{equation}
 where, for a fixed $\bi\in\mathbb N^{\overline C}$,  $y_\bi$ is the sequence $\big(T(x_\bi(\bar C),x_\bj(C)\big)_{\bj\in\mathbb N^C}$. 
 Since $r\geq q$, $\ell_r(Y)$ has cotype $r$ so that $\mathrm{id}:\ell_r(Y)\to\ell_r(Y)$ is $(r,1)$-summing. By the ideal property of summing operators,
 $\mathrm{id}:\ell_r(Y)\to\ell_s(Y)$ is still $(r,1)$-summing. By the inclusion theorem, this last map is $(s,\rho)$-summing, with
 $$\frac 1\rho=1-\frac 1r+\frac 1s=1-\sum_{j\in\bar C}\frac{1}{q_j^*}\in (0,1).$$
 Applying this to \eqref{eq:displike} yields
 $$\left(\sum_{\bi\in\mathbb N^m}\|T(x_\bi)\|^s\right)^{1/s}\leq \kappa \sup_{\varphi\in B_{[\ell_r(Y)]^*}}\left(\sum_{\bi\in\mathbb N^{\bar C}}|\varphi(y_\bi)|^\rho\right)^{1/\rho}.$$
 Observe that the constant $\kappa>0$ does not depend on $T$, but only on $Y$, $r$ and $\mathbf q$. We now apply Proposition \ref{prop:praciano} to get
 \begin{eqnarray*}
  \left(\sum_{\bi\in\mathbb N^m}\|T(x_\bi)\|^s\right)^{1/s}&\leq &\kappa  \sup_{\varphi\in B_{[\ell_r(Y)]^*}} \sup_{v\in\prod_{j\in\bar C}B_{\ell_{q_j}^*}}\sum_{\bi\in\mathbb N^{\bar C}}v_\bi\varphi(y_\bi)\\
  &\leq&\kappa   \sup_{v\in\prod_{j\in\bar C}B_{\ell_{q_j}^*}}\sup_{\varphi\in B_{[\ell_r(Y)]^*}}  \varphi\left(\left(T\left(\sum_{\bi\in\mathbb N^{\bar C}}v_\bi x_\bi(\bar C),x_{\bj}(C)\right)\right)_{\bj\in\mathbb N^C}\right)\\
  &\leq&\kappa\sup_{v\in\prod_{j\in\bar C}B_{\ell_{q_j}^*}}\left(\sum_{\bj\in\mathbb N^{C}}\left\|T\left(\sum_{\bi\in\mathbb N^{\bar C}}v_\bi x_\bi(\bar C),x_{\bj}(C)\right)\right\|^r\right)^{1/r}\\
  &\leq& \kappa \sup_{z\in X^{\bar C},\ \|z\|\leq 1}\left(\sum_{\bj\in\mathbb N^{C}}\|T(z,x_\bj(C))\|^r\right)^{1/r}
\end{eqnarray*}
since, for any $m\in\bar C$, by H\"older's inequality,
$$\left\|\sum_i v_i(m)x_i(m)\right\|=\sup_{x^*\in X_m^*}\sum_i v_i(m)\langle x^*,x_i(m)\rangle\leq 1.$$
\end{proof}

\begin{proof}[Proof of Theorem \ref{thm:main2}]
 We fix $k_0$ and $J$ satisfying the assumptions of the theorem. At the beginning we argue like in the proof of Theorem \ref{thm:main1}. Let $D=\bigcup_{j\in J}C_j$ and $z\in B_{X^D}$.
 We also set $C=\bar D$ and $C'=C\backslash\{k_0\}$. Let, for $j\in C$, $(x_i(j))\in X_j^{\mathbb N}$ with $w_{q_j}(x(j))\leq 1$. We can follow the arguments of the proof of Theorem \ref{thm:main1}
 up to the application of Lemma \ref{lem:popa} for the multilinear map $T^C(z)$.
This gives
$$\left(\sum_{\bi\in\mathbb N^{C_{k_0}}}\left(\sum_{\bj\in\mathbb N^{C'}}\|T(x_\bi(C_{k_0}),x_\bj(C'),z)\|^q\right)^{\frac{\gamma_{k_0,J}}q}\right)^{\frac 1{\gamma_{k_0,J}}}\leq\kappa.$$
Observe that the constant $\kappa$ does not depend on $z\in B_{X^D}$. Since $\gamma_{k_0,J}\geq q$, this implies
$$\left(\sum_{\bi\in\mathbb N^C}\|T(x_\bi(C),z)\|^{\gamma_{k_0,J}}\right)^{\frac1{\gamma_{k_0,J}}}\leq \kappa.$$
We may then apply Lemma \ref{lem:displike} to $T$ with $r=\gamma_{k_0,J}$ and 
$$\frac 1s=\frac 1r-\sum_{j\in D}\frac{1}{q_j^*}=\frac{1}{\gamma_{k_0,J}}-\sum_{j\in J}\frac{|C_j|}{p_j^*}$$
to get the conclusion.
\end{proof}

\begin{proof}[Proof of Theorem \ref{thm:main3}]
 The proof is completely similar but more elementary. Indeed, we can start from
 $$\left(\sum_{\bi\in\mathbb N^{C_k}}\|T(x_\bi(C_k),z)\|^{r_k}\right)^{\frac 1{r_k}}\leq \kappa$$
 for all $z\in\prod_{j\in\overline{C_k}}B_{X_j}$ and apply directly Lemma \ref{lem:displike} since $r_k\geq q$.
\end{proof}

\section{The inclusion theorem}
The proof of Theorem \ref{thm:inclusion} follows rather easily from Proposition \ref{prop:praciano}.
\begin{proof}[Proof of Theorem \ref{thm:inclusion}]
 We start from $x\in\prod_{k=1}^m B_{\ell_{q_k}^w(X_k)}$ and $u\in\prod_{k=1}^m B_{\ell_{\theta_k}}$ where $\frac 1{\theta_k}=\frac 1{p_k}-\frac 1{q_k}$. Then by
 H\"older's inequality, $ux=(u(1)x(1),\dots,u(m)x(m))$ belongs to $\prod_{k=1}^m B_{\ell_{p_k}^w(X_k)}$. Hence,
 $$\left(\sum_{\bi\in\mathbb N^m}|u_\bi|^r\|T(x_\bi)\|^r\right)^{1/r}\leq \pi_{r,p}^{\textrm{mult}}(T).$$
 We may then apply Proposition \ref{prop:praciano} to the multilinear form $A:\ell_{\frac{\theta_1}r}\times\cdots\times\ell_{\frac{\theta_m}r}\to\mathbb C$
 defined by $A(v)=\sum_{i\in\mathbb N^m}v_{\bi}\|T(x_\bi)\|^r$. This is possible since
 $$1-\sum_{j=1}^m \frac{r}{\theta_j}=r\left(\frac 1r-\sum_{j=1}^m \frac{1}{p_j}+\sum_{j=1}^m \frac 1{q_j}\right)>0.$$
 This yields immediately Theorem \ref{thm:inclusion}.
\end{proof}

Of course, it is natural to compare P\'erez-Garc\'ia result with ours. If we start from a $(p,p)$-summing multilinear map, the former is better. But if we start from a multiple $\left(\frac{2m}{m+1},1\right)$-summing
$m$-linear map, Theorem \ref{thm:inclusion} shows that, for any $s\in\left(\frac{2m}{m+1},2\right)$, it is also multiple $\left(s,\frac{2m^2s}{2m+(2m^2-m-1)s}\right)$-summing whereas
we cannot expect from P\'erez-Garc\'ia theorem a better result than it is $(s,s)$-summing. It is easy to check that for those $s$, 
$$\frac{2m^2 s}{2m+(2m^2-m-1)s}<s.$$
In other words, Theorem \ref{thm:inclusion} gives a better conclusion. Applications of Theorem \ref{thm:inclusion} are given in \cite{PSST17}.

\section{Applications to harmonic analysis}

\subsection{Product of $p$-Sidon sets}
\begin{proof}[Proof of Theorem \ref{thm:sidon}]
Let $G=G_1\times\cdots\times G_m$ and $f=\sum_{\bi\in\mathbb N^m}a_\bi \gamma_\bi$ be a polynomial with spectrum in $\Lambda_1\times\cdots\times\Lambda_m$. Here $\gamma_\bi$ denotes the tensor product $\gamma_{i_1}(1)\otimes\cdots\otimes \gamma_{i_m}(m)$ and each $\gamma_{i_j}(j)$ belongs to $\Gamma_j$. Fix $k\in\{1,\dots,m\}$, let $C_k=\{k\}$, $\widehat{G_k}=G_1\times\cdots G_{k-1}\times G_{k+1}\times\cdots\times G_m$ and  $\widehat{\Lambda_k}=\Lambda_1\times\cdots \Lambda_{k-1}\times \Lambda_{k+1}\times\cdots\times \Lambda_m$. It is well-known that the product of $\Lambda(2)$-sets is still a $\Lambda(2)$-set (this follows from Minkowski's inequality for integrals). Hence, $\widehat{\Lambda_k}$ is a $\Lambda(2)$-set and we deduce that for any $i\in\mathbb N=\mathbb N^{C_k}$, 
$$\left(\sum_{\bj\in\mathbb N^{\overline{C_k}}}|a_{i,\bj}|^2\right)^{p_k/2}\leq \kappa 
\int_{\widehat{G_k}}\left|\sum_{\bj\in\mathbb N^{C_k}}a_{i,\bj}\gamma_\bj (g')\right|^{p_k}dg'.$$
We sum over $i\in\mathbb N^{C_k}$ and we use that $\Lambda_k$ is $p_k$-Sidon to deduce that
\begin{eqnarray*}
\left(\sum_{i\in\mathbb N^{C_k}}
\left(
\sum_{\bj\in\mathbb N^{\overline{C_k}}}
|a_{i,\bj}|^2\right)^{p_k/2}
\right)^{1/p_k}
&\leq&\kappa \left(\int_{\widehat{G_k}}\sup_{g\in G_k}|f(g,g')|^{p_k}dg'\right)^{1/p_k}\\
&\leq&\kappa\|f\|_\infty.
\end{eqnarray*}
The result now follows from Lemma \ref{lem:popa}. We postpone the proof of optimality to the last section.
\end{proof}

\subsection{Product of $p$-Rider sets}
Beyond $p$-Sidon sets, L. Rodr\'iguez-Piazza	 has introduced in \cite{RP87} another class of sets extending naturally that of Sidon sets. For $G$ a compact abelian group with dual $\Gamma$, a subset $\Lambda\subset \Gamma$ is called \emph{$p$-Rider} ($1\leq p<2$) if there is a constant $\kappa>0$ such that  each $f\in\mathcal C(G)$ with $\hat f$ supported on $\Lambda$ satisfies
$$\|\hat f\|_{\ell_p}\leq \kappa [\![f]\!]:=\int_{\Omega}\left \|\sum_{\gamma\in\Gamma}\veps_\gamma \hat f(\gamma)\gamma\right\|_\infty d\mathbb P$$
where $(\veps_\gamma)_{\gamma\in\Gamma}$ is a sequence of independent Bernoulli variables. The terminology $p$-Rider comes from Rider's theorem which asserts that $1$-Sidon sets and $1$-Rider sets coincide. Observe that it is easy to prove that a $p$-Sidon set is always a $p$-Rider set (see \cite{LRP03}), but the converse is an open question.

It turns out that $p$-Rider sets are usually easier to manage than $p$-Sidon sets. This is due to the inconditionnality of the norm $[\![\cdot]\!]$. For instance, this last property implies immediately that the union
of two $p$-Rider sets is still a $p$-Rider set, a fact which is unknown for $p$-Sidon sets. This is also the case for the direct product.

\begin{theorem}\label{thm:rider}
Let $G_1,\dots,G_m$, $m\geq 2$, be compact abelian groups with respective dual groups $\Gamma_1,\dots,\Gamma_m$. For $1\leq j\leq m$, let $\Lambda_j\subset\Gamma_j$ be a $p_j$-Rider set. Then $\Lambda_1\times\cdots\times\Lambda_m$ is a $p$-Rider set in $\Gamma_1\times\cdots\times \Gamma_m$ for
$$\frac1p=\frac 12+\frac{1}{2R}\textrm{ and }R=\sum_{k=1}^m \frac{p_k}{2-p_k}.$$
\end{theorem}
This result was already proved in \cite{RP91} using an arithmetical characterization of $p$-Rider sets. We provide a new (and maybe more elementary) proof using our machinery.
\begin{proof}
Let $G=G_1\times\cdots\times G_m$ and $f=\sum_{\bi\in\mathbb N^m}a_\bi \gamma_\bi$ be a polynomial with spectrum in $\Lambda_1\times\cdots\times\Lambda_m$. Fix $k\in\{1,\dots,m\}$ and keep the notations of the proof of Theorem \ref{thm:sidon}. Let $(\Omega,\mathcal A,\mathbb P)$ be a probability space and consider
three sequences $(\veps_{i,\bj})_{i\in\mathbb N,\bj\in\mathbb N^{\overline{C_k}}}$, $(\delta_\bj)_{\bj\in\mathbb N^{C_k}}$, $(\eta_i)_{i\in\mathbb N}$ of independent Bernoulli variables on $(\Omega,\mathcal A,\mathbb P)$. Then, for any $i\in\mathbb N=\mathbb N^{C_k}$ and any $\omega\in\Omega$, by the Khintchine inequalities, 
\begin{eqnarray*}
\left(\sum_{\bj\in\mathbb N^{\overline{C_k}}}|a_{i,\bj}|^2\right)^{p_k/2}&=&\left(\sum_\bj |a_{i,\bj}\veps_{i,\bj}(\omega)|^2\right)^{p_k/2}\\
&\leq&\kappa_1 \int_{\Omega}\left|\sum_{\bj}a_{i,\bj}\veps_{i,\bj}(\omega)\delta_{\bj}(\omega')\right|^{p_k}d\mathbb P(\omega').
\end{eqnarray*}
We sum over $i$ and use that $\Lambda_k$ is a $p_k$-Rider set to get
\begin{eqnarray*}
&&\sum_{i\in\mathbb N}\left(\sum_{\bj\in\mathbb N^{\overline{C_k}}}|a_{i,\bj}|^2\right)^{p_k/2}\\
&&\leq
\kappa_2 \int_\Omega\left(\int_{\Omega} \sup_{g\in G_k}\left|\sum_{i,\bj}a_{i,\bj}\veps_{i,\bj}(\omega)\delta_{\bj}(\omega')\eta_i(\omega'')\gamma_i(k)(g)\right|d\mathbb P(\omega'')\right)^{p_k}d\mathbb P(\omega')\\
&&\leq\kappa_3 \int_\Omega\int_{\Omega} \sup_{g\in G_k}\left|\sum_{i,\bj}a_{i,\bj}\veps_{i,\bj}(\omega)\delta_{\bj}(\omega')\eta_i(\omega'')\gamma_i(k)(g)\right|^{p_k}d\mathbb P(\omega'')d\mathbb P(\omega')
\end{eqnarray*}
where the last line comes from Kahane's inequalities. We then integrate over $\omega\in\Omega$, exchange integrals, apply the contraction principles to Bernoulli variables (see \cite[Proposition 12.2]{DJT}) and use a last time Kahane's inequality to get
\begin{eqnarray*}
\sum_{i\in\mathbb N}\left(\sum_{\bj\in\mathbb N^{\overline{C_k}}}|a_{i,\bj}|^2\right)^{p_k/2}&\leq&
\kappa_3\int_{\Omega}\sup_{g\in G_k}\left|\sum_{i,\bj}a_{i,\bj}\veps_{i,\bj}(\omega)\gamma_i(k)(g)\right|^{p_k}d\mathbb P(\omega)\\
&\leq&\kappa_4 [\![f]\!]^{p_k}.
\end{eqnarray*}
We conclude using Lemma \ref{lem:popa}.
\end{proof}

\section{About the optimality}
\subsection{Optimality for coordinatewise summability}
We now discuss the optimality of our results. We first show that Theorem \ref{thm:intro} is optimal when we restrict ourselves to cotype 2 spaces and $1\leq p\leq 2$.
\begin{theorem}\label{thm:opti1}
Let $p\in [1,2]$, $r\geq p$ satisfying $\frac 1r\geq\frac 1p-\frac 12$ and $m\geq 1$. Then the optimal $s$ such that every $m$-linear map $T:X_1\times\cdots\times X_m\to\ell_2$ which is separately $(r,p)$-summing is automatically $(s,p)$-summing satisfies
\begin{itemize}
\item $\frac 1s=\frac{m-1}{2m}+\frac1{mr}-\frac {m-1}{p^*}$ provided $\frac 1r-\frac{m-1}{p^*}>\frac 12$;
\item $\frac 1s=\frac 1r-\frac{m-1}{p^*}$ provided $0<\frac 1r-\frac{m-1}{p^*}\leq\frac 12$.
\end{itemize}
\end{theorem}
It should be observed that the assumption $\frac1r\leq\frac 1p-\frac 12$ is not a restriction on the possible values of $r$. Indeed, a linear map with values in a cotype $2$ space is always $(2,1)$-summing, hence $(r,p)$-summing with $\frac1r=\frac 1p-\frac 12$.
\begin{proof}
We shall use the following result proved partly in \cite{DiSP14} and partly in \cite{ABPS14}. Let $1\leq u\leq 2$. Define $\rho$ as the best (=smallest) real number such that, for all $m$-linear maps $A:\ell_{p^*}\times\cdots\times \ell_{p^*}\to \ell_u$, the composition $I_{u,2}\circ A$ is multiple $(\rho,p)$-summing where $I_{u,2}$ denotes the identity map from $\ell_u$ into $\ell_2$. Then
\begin{itemize}
\item $\frac 1\rho=\frac 12+\frac 1m\left(\frac 1u-\frac 12-\frac{m}{p^*}\right)$ provided $0<\frac m{p^*}<\frac 1u-\frac12$;
\item $\frac 1\rho=\frac 1u-\frac m{p^*}$ provided $\frac 1u-\frac 12\leq \frac m{p^*}<\frac 1u$.
\end{itemize}
The real numbers $r$ and $p$ being fixed (and satisfying the assumptions of Theorem \ref{thm:opti1}), we fix $u\in [1,2]$ such that $\frac 1r=\frac 1u+\frac 1p-1$. By the Bennett-Carl inequalities, $I_{u,2}$ is $(r,p)$-summing with $\frac 1r=\frac 1u+\frac 1p-1$ so that $I_{u,2}\circ A$ is separely $(r,p)$-summing. Then the optimal $s$ in Theorem \ref{thm:opti1} must satisfy $s\geq \rho$. But using the relation linking $u$, $p$ and $r$, it is easy to see that the condition $\frac{m}{p^*}<\frac 1u-\frac 12$ is equivalent to $\frac{1}{r}-\frac{m-1}{p^*}>\frac 12$ and that the values of $\frac 1\rho$ are exactly the optimal values appearing in Theorem \ref{thm:opti1}.
\end{proof}

\subsection{Optimality for the inclusion theorem}
We now show that, in full generality, Theorem \ref{thm:inclusion} is also optimal.
\begin{theorem}\label{thm:opti2}
 Let $r\geq 2$ and $p=\frac{2r}{r+1}$. Then there exists a bilinear form $T:\ell_2\times\ell_2\to\mathbb C$ which is $(r,p)$-summing and such that, for every $s\geq 2$ and $q\geq p$, it is $(s,q)$-summing
 if and only if 
$$\frac 1s-\frac2{q}\leq\frac{1}{r}-\frac 2p.$$
\end{theorem}
\begin{proof}
 Let $T(x,y)=\sum_{i=1}^{+\infty}x_iy_i$, which has norm 1. Then by Corollary \ref{cor:praciano}, as all bilinear forms, $T$ is $(r,p)$-summing. 
 Conversely, let us assume that it is also $(s,q)$-summing.
 We choose $x=(e_i)_{i=1,\dots,n}$ so that $w_q(x)=n^{\max\left(\frac 1q-\frac 12,0\right)}$. For this choice we get
 $$n^{\frac 1s}=\left(\sum_{i,j=1}^n |T(e_i,e_j)|^s\right)^{\frac 1s}\leq \pi_{s,q}(T)w_q(x)^2\leq \pi_{s,q}(T)n^{\max\left(\frac 2q-1,0\right)}.$$
This implies $q\leq 2$ and $\frac 1s\leq \frac{2}q-1$ namely
$$\frac 1s-\frac 2q\leq \frac 1r-\frac 2p.$$
\end{proof}
In view of this example and P\'erez-Garc\'ia's result, it seems conceivable that something similar does not happen if we start with $r\leq s\leq 2$. This deserves further investigation.

\subsection{Optimality for the product of $p$-Sidon sets}
We finally conclude by proving the optimality of Theorem \ref{thm:sidon}. To simplify the notations, we will only prove it for the product of two sets. We shall work with $G=\Omega=\{-1,1\}^{\mathbb N}$ whose dual group $\Gamma$ is the set of Walsh functions. Recall that if $(r_n)_{n\in\mathbb N}$ is the sequence of Rademacher functions on $\Omega$, defined by $r_n(\omega)=\omega_n$, $\omega\in\Omega$, then the Walsh functions
are the functions $w_A=\prod_{n\in A}r_n$ where $A$ is any finite subset of $\mathbb N$ (in particular, $w_\varnothing=1$). We will prove the following theorem, which clearly implies optimality in Theorem \ref{thm:sidon}.
\begin{theorem}
Let $\Omega=\{-1,1\}^{\mathbb N}$, $\Gamma$ its dual group, $p_1$, $p_2$ rational numbers in $[1,2)$. There exist two subsets $\Lambda_1$, $\Lambda_2$ of $\Gamma$ which are respectively $p_1$-Sidon or $p_2$-Sidon, and such that their direct product $\Lambda_1\times\Lambda_2$ is not $p$-Sidon for 
$$\frac 1p>\frac 12+\frac1{2R}\textrm{ where }R=\frac{p_1}{2-p_1}+\frac{p_2}{2-p_2}.$$
\end{theorem}
\begin{proof}
The proof needs some preparation. First we recall a necessary condition for a subset $\Lambda\subset\Gamma$ to be $p$-Sidon (see \cite[Theorem VII.41]{Bl01}):
\begin{lemma}\label{lem:necpsidon}
Let $\Lambda\subset\Gamma$ and assume that $\Lambda$ is $p$-Sidon. Then there exists $\kappa>0$ such that, for any polynomial $f$ supported on $\Lambda$, for any $s\geq 1$, 
$$\frac{\|f\|_{L^s}}{\sqrt s\|\hat f\|_{\frac{2p}{3p-2}}}\leq\kappa.$$
\end{lemma}
We write $p_1=\frac{2m_1}{m_1+k_1}$ and $p_2=\frac{2m_2}{m_2+k_2}$. Let $S_1^{1},\dots,S_{n_1}^1$ (resp. $S_1^2,\dots,S_{n_2}^2$) the subsets of $\{1,\dots,m_1\}$ (resp. of $\{1,\dots,m_2\}$) with cardinal $k_1$ (resp. $k_2$). Let $E_1^1,\dots,E_{n_1}^1,E_1^2,\dots,E_{n_2}^2$ be pairwise disjoint infinite subsets of the Rademacher system and enumerate each $E_l^\delta$, $\delta\in\{0,1\}$, $l\in\{1,\dots,n_\delta\}$ by $\mathbb N^{k_\delta}$:
$$E_l^\delta=\left\{\gamma_{l,\bj}^\delta;\ \bj\in\mathbb N^{k_\delta}\right\}.$$
Define $\Pi_{S_l^\delta}$ as the projection from $\{1,\dots,m_\delta\}$ onto $S_l^\delta$. We finally consider 
$$\Lambda_\delta=\left\{\gamma^\delta_{1,\Pi_{S_1^\delta}\bj}\cdots \gamma^\delta_{n_\delta,\Pi_{S_{n_\delta}^\delta}\bj};\ \bj\in\mathbb N^{m_\delta}\right\}.$$
It is shown in \cite[p. 465]{Bl01} that $\Lambda_\delta$ is $p_\delta$-Sidon (and nothing better!). We shall prove that $\Lambda_1\times\Lambda_2$ is not $p$-Sidon for 
$$\frac 1p>\frac 12+\frac1{2R}\textrm{ where }R=\frac{p_1}{2-p_1}+\frac{p_2}{2-p_2},$$
namely 
$$\frac 1p>\frac{m_1k_1+m_2k_1+k_1k_2}{2(m_1k_1+m_2k_2)}.$$
To do this, we consider $N$ a large integer and set $N_1=N^{k_2}$ and $N_2=N^{k_1}$ so that $N_1^{k_1}=N_2^{k_2}$. We then define
$$f_N=\sum_{\stackrel{\bj\in\{1,\dots,N_1\}^{m_1}}{\bk\in\{1,\dots,N_2\}^{m_2}}}\gamma^1_{1,\Pi_{S_1^1} \bj}\cdots \gamma^1_{n_1,\Pi_{S_{n_1}^1}\bj}\gamma^2_{1,\Pi_{S_1^2} \bk}\cdots \gamma^2_{n_2,\Pi_{S_{n_2}^2}\bk}$$
which is a polynomial supported on $\Lambda_1\times\Lambda_2$, and the Riesz product
$$R_N=\prod_{l=1}^{n_1}\prod_{\bj\in\{1,\dots,N_1\}^{k_1}}(1+\gamma_{l,\bj}^1)\times\prod_{l=1}^{n_2}\prod_{\bj\in\{1,\dots,N_2\}^{k_2}}(1+\gamma_{l,\bj}^2).$$
Then $\|R_N\|_1=\int R_N=1$ (recall that $R_N$ is positive) whereas $\|R_N\|_2=2^{n_1+N_1^{k_1}+n_2+N_2^{k_2}}=2^{n_1+n_2+2N^{k_1k_2}}$. By interpolation, for any $s>2$, 
$$\|R_N\|_{s^*}\leq 2^{\frac{n_1+n_2+2N^{k_1k_2}}{s}}.$$
On the other hand, by the very definition of $R_N$, $R_N=f_N+Q_N$ where the spectrum of $Q_N$ is disjoint from that of $f_N$. Hence,
$$\int_{\Omega\times\Omega}R_Nf_N=\int_{\Omega\times\Omega}f_N^2=\sum_{\bj,\bk}1^2=N_1^{m_1}N_2^{m_2}=N^{m_1k_2+m_2k_1}.$$
Now, observe that Holder's inequality also yields
$$\left|\int_{\Omega\times\Omega}R_Nf_N\right|\leq \|R_N\|_{s^*}\|f_N\|_s\leq 2^{\frac{n_1+n_2+2N^{k_1k_2}}{s}}\|f_N\|_s.$$
We choose $s=N^{k_1k_2}$ so that one obtains 
$$\|f_N\|_s\geq \kappa N^{m_1k_2+m_2k_1}.$$
In order to apply Lemma \ref{lem:necpsidon} we just compute
$$\|\widehat{f_N}\|_{\frac{2p}{3p-2}}=\left(N_1^{m_1}N_2^{m_2}\right)^{\frac{3p-2}{2p}}=N^{(m_1k_2+m_2k_1)\frac{3p-2}{2p}}.$$
Thus,
$$\frac{\|f_N\|_{L^s}}{\sqrt s\|\hat f_N\|_{\frac{2p}{3p-2}}}\geq\kappa N^{(m_1k_2+m_2k_1)\frac{2-p}{2p}-\frac{k_1k_2}2}.$$
If $\Lambda_1\times\Lambda_2$ is $p$-Sidon, then Lemma \ref{lem:necpsidon} tells us that 
$$\left(\frac 12-\frac 1p\right)(m_1k_2+m_2k_1)-\frac{k_1k_2}2\leq 0$$
which is exactly the desired inequality.
\end{proof}

\providecommand{\bysame}{\leavevmode\hbox to3em{\hrulefill}\thinspace}
\providecommand{\MR}{\relax\ifhmode\unskip\space\fi MR }
\providecommand{\MRhref}[2]{%
  \href{http://www.ams.org/mathscinet-getitem?mr=#1}{#2}
}
\providecommand{\href}[2]{#2}

\end{document}